\documentclass[12pt]{amsart}
\usepackage{amscd, latexsym}
\usepackage{amsmath, amsthm, amstext, amsbsy, amssymb}
\usepackage{mathrsfs}

\usepackage{graphicx}
\graphicspath{{images/}}

\usepackage{tikz} 
\usepackage{float}
\usetikzlibrary{decorations.pathmorphing}
\usetikzlibrary{decorations.pathreplacing}
\usetikzlibrary{decorations.markings}
\usetikzlibrary{arrows}
\usetikzlibrary{arrows.meta}
\usetikzlibrary{quotes}
\usetikzlibrary{positioning}
\usetikzlibrary{calc}
\usetikzlibrary{bending}
\usetikzlibrary{shapes.geometric}
\usetikzlibrary{patterns}
\usetikzlibrary{patterns.meta}
\usetikzlibrary{fit}
\usetikzlibrary{decorations}
\usetikzlibrary{fadings}
\usetikzlibrary{bending}
\usetikzlibrary{intersections}
\usetikzlibrary{backgrounds}

\usepackage{xcolor}                 
\usepackage[shortlabels]{enumitem} 

\usepackage{hyperref} 				
\hypersetup{						
	colorlinks=true,
	linkcolor=blue,
	citecolor=red,
}

\allowdisplaybreaks		   			

\newtheorem{definition}{Definition}
\newtheorem{prop}{Proposition}
\newtheorem{lem}{Lemma}
\newtheorem{thrm}{Theorem}

\newtheorem{rmk}{Remark}
\newtheorem{exm}{Example}

\renewcommand{\le}{\leqslant}
\renewcommand{\ge}{\geqslant}
\renewcommand{\leq}{\leqslant}

\let\intt\int
\renewcommand{\int}{\intt\limits}

\newcommand{\C}{\mathbb{C}} 						
\newcommand{\R}{\mathbb{R}}		                    
\newcommand{\N}{\mathbb{N}}	                	    
\newcommand{\Z}{\mathbb{Z}}	 					    
\newcommand{\T}{\mathrm{T}}                        
\newcommand{\F}{\mathcal{F}}                            
\renewcommand{\H}{\mathcal{H}}                          
\newcommand{\ZZ}{\mathcal{Z}}                           

\renewcommand{\a}{\alpha}									
\renewcommand{\b}{\beta}									
\newcommand{\de}{\delta}									
\newcommand{\e}{\varepsilon}								
\renewcommand{\k}{\kappa}									
\renewcommand{\l}{\lambda}									
\renewcommand{\L}{\Lambda}									
\let\originalnu\nu											
\renewcommand{\nu}{\originalnu} 					        
\newcommand{\f}{\varphi}									

\DeclareMathOperator{\Cl}{Cl}							    
\DeclareMathOperator{\dist}{dist}							

\renewcommand{\Im}{\operatorname{Im}}		    	

\DeclareMathOperator{\Span}{span}							

\title[Exponential Riesz bases in $L^2$ on two intervals]
{Exponential Riesz bases in $L^2$ on two intervals}
\author{Yurii Belov, Mikhail Mironov}
\address{
\phantom{x}\,\, Yurii Belov,
\newline Department of Mathematics and Computer Science, St.~Petersburg State University, St. Petersburg, Russia,
\newline {\tt j\_b\_juri\_belov@mail.ru}
\smallskip
\newline \phantom{x}\,\, Mikhail Mironov,
\newline Department of Mathematics and Computer Science, St.~Petersburg State University, St.~Petersburg, Russia,
\newline {\tt mironovv3@yandex.ru}
\smallskip
}

\begin{document}

\maketitle

\begin{abstract} 

{
We give sufficient conditions for the exponential system to be a Riesz basis in $L^2(E)$, where $E$ is a union of two intervals. We show that these conditions are close to be necessary. In addition, we demonstrate ``extra point effect'' for such systems, i.e. it may happen that the Riesz basis in $L^2(E)$ differs by one point from the Riesz basis on an interval.
}

\end{abstract}

\section{Introduction}

{\subsection{Exponential systems}\footnote{The work was supported by the Russian Science Foundation grant 19-11-00058P} Let $S$ be a bounded subset of $\R$. We would like to present an arbitrary function $f\in L^2(S)$ in the form of a series 
\begin{equation}
\sum_{\lambda\in\Lambda}c_\lambda e^{i\lambda t},
\label{meq}
\end{equation}
where $\Lambda$ is some given discrete set. For example, if $S=[-\pi,\pi]$, $\Lambda=\mathbb{Z}$, then we get the classical orthogonal Fourier series on an interval, where the coefficients $c_\lambda$ are unique and defined by $c_n=(f,e^{i nt})$, $n\in\mathbb{Z}$.

However, orthogonal in $L^2(S)$ exponential bases are rare. On the other hand, we can omit the requirement of orthogonality and require only the uniqueness of the expansion \eqref{meq}.   

\begin{definition} We will say that the system $\{e^{i\lambda t}\}_{\lambda\in\Lambda}$ is an unconditional basis if the expansion \eqref{meq} converges in the norm for any order of terms , and the coefficients $c_\lambda$ are unique.
\end{definition}

It is well known that a system of vectors $\{e^{i\lambda t}\}_{\lambda\in\Lambda}$ is an unconditional basis if and only if it is complete and there exist $A,B>0 $ such that
\begin{equation}
A\sum_{\lambda \in \L}{|c_\lambda|^2}\leq\biggl{\|}\sum_{\lambda\in\Lambda}c_\lambda \frac{e^{i\lambda t}}{\|e^{i\lambda t}\|}\biggr{\|}^2\leq B\sum_{\lambda\in \L }{|c_\lambda|^2},
\label{Req}
\end{equation}
for any finite sequence $\{c_\lambda\}$. A system of vectors that satisfies the inequality \eqref{Req} is called \textit{a Riesz sequence}. A complete Riesz sequence is called \textit{a Riesz basis}.

\medskip

If $S$ is an interval, then a complete description of the Riesz exponential bases is known. The key step was taken by B.S. Pavlov in 1979 \cite{Pavlov}. He managed to describe all the Riesz bases lying in some strip $\sup_{\lambda\in\Lambda}|\Im \lambda|<\infty$ . Later N.K. Nikol'skii \cite{Nikol'skii} generalized this result to sequences lying in the half-plane. 
In 1992 A.M. Minkin \cite{Minkin} got rid of any restrictions.

\medskip

If $S$ is the union of several intervals, then the task becomes much more difficult.
{
Even the existence of some Riesz basis is a highly involved question.   
}
A. Kohlenberg \cite{Kohlenberg} constructed the first example of a Riesz basis of exponentials for two intervals of equal length, L. Bezuglaya and V. Katsnel'son \cite{BezKats} constructed an exponential basis for two intervals with integer ends, K. Seip \cite{ Seip} has constructed a Riesz basis for an arbitrary union of two intervals and some special cases of union of a finite number of intervals. Y. Lyubarskii and I. Spitkovsky \cite{LyubSpit} constructed a basis for any number of intervals, but with not necessarily real $\lambda$. Finally, G. Kozma and S. Nitzan \cite{KozNit} constructed a {\it real} Riesz basis for an arbitrary union of a finite number of intervals. Note that G. Kozma, S. Nitzan, and A. Olevskii \cite{KozNitOlev} recently proved that there exists a bounded set $S$ for which there is no real Riesz basis of exponents.

\medskip

At the moment, all results for disconnected sets $S$ are of a particular nature --- they are concrete examples of Riesz bases. { We are able to find some sufficient basis conditions}, see Theorem \ref{thrm: suff}, for two intervals that are close to necessary, see Theorems \ref{thrm: ness T}, \ref{thrm: ness L T}, \ref{thrm: ness F+}. In particular, our conditions allow us to demonstrate the ``extra point'' effect when gluing intervals, see Example \ref{prop: concrete examples}, and Theorem \ref{thrm: ness +-1}.

\subsection{Paley-Wiener spaces} Often, instead of the space $L^2(S)$, it is convenient to study its Fourier image (Paley-Wiener space). It is well known that Riesz bases of exponents $\{e^{i\lambda t}\}_{\lambda\in\Lambda}$ correspond one-to-one to complete interpolating sequences for Paley-Wiener space. Therefore, we will study complete interpolating sequences for Paley-Wiener spaces on two intervals $PW_E$, $$E = [-\pi, a] \cup [b, \pi].$$

}

The Paley-Wiener space $PW_E$ is the space of finite energy signals whose frequency spectrum is contained in the set $E$.

    \begin{definition}
        The Fourier transform $\F$ of the function $f$ is defined as
        \begin{equation*}
            \F f(t) = \frac{1}{2 \pi} \int_\R f(x) e^{-ixt} dx.  
        \end{equation*}
    \end{definition}
    
    \begin{definition}
        Let $S$ be a compact subset of $\R$. The Paley-Wiener space on a set $S$ is $PW_S = \F^{-1} L^2(S)$.
    \end{definition}
    
    In what follows, instead of $PW_{[-a, a]}$ we will write $PW_a$.
    
    {
    The space $PW_S$ is a reproducing kernel Hilbert space. The reproducing kernel $\k^S_\l$ is given by
    \begin{equation*}
        \k^S_\l(z) = \int_S e^{i (z - \bar \l) t} dt.   
    \end{equation*}  
    }
    

    { 
    \begin{definition}
        A sequence $\L$ is called a complete interpolating sequence for $PW_S$ if for any $l^2$ data $\{ a_\l \}_{\l \in \L}$ there is a unique solution of the interpolation problem
        \begin{equation*}
            f \in PW_S, \ \frac{f(\l)}{|| \k^S_\l ||} = a_\l, \ \l \in \L.
        \end{equation*}             
    \end{definition}
    }
    
    It is well known that $\L$ is a complete interpolating sequence for $PW_S$ if and only if the system of reproducing kernels $\{ \k^S_\l \}_{\l \in \L}$ is a Riesz basis in $PW_S$.

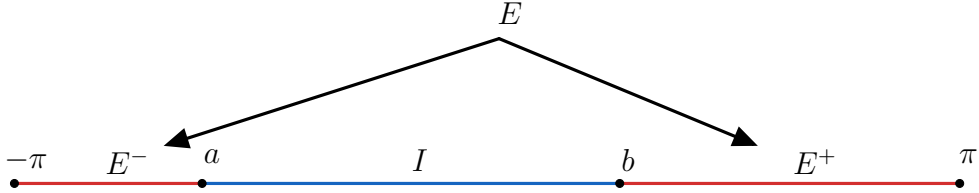
    \definecolor{rvwvcq}{rgb}{0.08235294117647059,0.396078431372549,0.7529411764705882}
    \definecolor{dtsfsf}{rgb}{0.8274509803921568,0.1843137254901961,0.1843137254901961}
    \begin{center}
        \begin{tikzpicture}[line cap=round,line join=round,>=triangle 45,x=2.0cm,y=2.0cm]
            \clip(-3.5,-0.3) rectangle (3.5,1.3);
            \draw [line width=1.2pt,color=dtsfsf] (-3.141592653589793,0.)-- (-1.8941112249126484,0.);
            \draw [line width=1.2pt,color=dtsfsf] (3.141592653589793,0.)-- (0.8818004589799778,0.);
            \draw [line width=1.2pt,color=rvwvcq] (-1.8941112249126484,0.)-- (0.8818004589799778,0.);
            \draw (-2.6085365561983984,0.2987700235156789) node[anchor=north west] {$E^-$};
            \draw (-3.2806102068890186,0.2987700235156789) node[anchor=north west] {$- \pi$};
            \draw (3.0660475699158,0.2987700235156789) node[anchor=north west] {$\pi$};
            \draw (-1.9554838578858145,0.2987700235156789) node[anchor=north west] {$a$};
            \draw (0.8215751893075026,0.2987700235156789) node[anchor=north west] {$b$};
            \draw (-0.5732946517485015,0.2987700235156789) node[anchor=north west] {$I$};
            \draw (1.9628323319896877,0.2987700235156789) node[anchor=north west] {$E^+$};
            \draw (-0.0026660804074089034,1.2751789122548822) node[anchor=north west] {$E$};
            \draw [->,line width=1.2pt] (0.07987083192761656,0.9633621131506843) -- (-2.15,0.25);
            \draw [->,line width=1.2pt] (0.07987083192761656,0.9633621131506843) -- (1.8,0.25);
            \begin{scriptsize}
            \draw [fill=black] (-3.141592653589793,0.) circle (1.5pt);
            \draw [fill=black] (-1.8941112249126484,0.) circle (1.5pt);
            \draw [fill=black] (3.141592653589793,0.) circle (1.5pt);
            \draw [fill=black] (0.8818004589799778,0.) circle (1.5pt);
            \end{scriptsize}
        \end{tikzpicture}        
    \end{center}

    Let $-\pi < a < b < \pi$, $I = [a, b]$, $E^- = [-\pi, a]$, $E^+ = [b, \pi]$, $E = E^- \cup E^+$ , and $E^g = E^- + E^+ = [-\pi + b, a + \pi]$ --- glued $E^-$ and $E^+$. 
    
    To study complete interpolating sequences for $PW_E$, we will need to associate each sequence $\L$ with some class of entire functions.
    \begin{definition}
        A generating function for a sequence $\L$ in $PW_E$ is a function $F$ such that $F(\l) = 0, \ \l \in \L$ and
        \begin{equation*}
            F \in PW_E + z PW_E.
        \end{equation*}
    \end{definition}
    In particular, such $F$ can be decomposed into a sum
    \begin{equation*}
        F(z) = F^-(z) + F^+(z),
    \end{equation*}
    where
    \begin{equation*}
        F^\pm \in PW_{E^\pm} + z PW_{E^\pm}.
    \end{equation*}
    
    It is clear that if $F$ is generating for a sequence $\L$ in $PW_E$, then in addition to $\L$ $F$ may have some other zeros. In what follows, we will denote the set of these zeros by $\T$. Throughout the paper we will assume that all points in $\T$ are simple. This assumption is solely for the reader's convenience. The results obtained in this paper remain valid even in the presence of points with multiplicities in $\T$.
    
    { For the case of one interval, there exists a unique up to a constant multiple generating function, and this function has no other zeros.
    
    As we will see later, for any complete interpolating sequence for $PW_E$ there exists at least one generating function in $PW_E$, $E = E^- \cup E^+$.
    }
   
   \subsection{Main results}
   

    {We start with sufficient conditions for a sequence $\L$ to be complete interpolating for $PW_E$.}
    
    \begin{thrm} \label{thrm: suff}
        Let $F$ be a generating function for $\L$ in $PW_E$, and $\T$ be the other zeros of $F$. \\
        Suppose that:
        \begin{enumerate}[label = (\roman*)]
            \item $\L \cup \T$ is a complete interpolating sequence for $PW_\pi$,
            \item $\T$ is a complete interpolating sequence for $PW_I$,
            \item 
            \begin{equation} \label{eq: F^+ > F'}
                \inf_{t \in \T} \left| \frac{F^+(t)}{F'(t)} \right| || \k^{\pi}_{t} ||_2 ||\k^I_{\bar t}||_2 > 0.
            \end{equation}
        \end{enumerate}   
        Then $\L$ is a complete interpolating sequence for $PW_E$.
    \end{thrm}
    \begin{rmk}
        If $\T$ lies in a horizontal strip $(\sup_{t \in \T} |\Im t| < \infty)$, then \eqref{eq: F^+ > F'} comes down to
        \begin{equation*}
            \inf_{t \in \T} \left| \frac{F^+(t)}{F'(t)} \right| > 0.
        \end{equation*}
    \end{rmk}
    
    { We will show that the technical condition $(iii)$ is necessary (see Theorem \ref{thrm: ness F+}). Without condition $(iii)$ Theorem \ref{thrm: suff} can be simply stated as follows \\
    \textit{If the complete interpolating sequence for $PW_I$ is removed from the complete interpolating sequence for $PW_\pi$, then what remains is the complete interpolating sequence for $PW_E$.}
    
    Though, condition $(iii)$ cannot be omitted, see Remark \ref{rmk: crucial (iii)}.
    }
    
    Using Theorem \ref{thrm: suff}, we construct a wide class of examples of complete interpolating sequences for Paley-Wiener spaces on two intervals.
    
    Moreover, Theorem \ref{thrm: suff} allows us to construct examples of complete interpolating sequences for $PW_E$, which differ by one point from the complete interpolating sequences for $PW_{E^g}$.
    
    \begin{exm} \label{prop: concrete examples}
        Let $E = [-\pi, -\pi/2] \cup [\pi/2, \pi]$ and
        \begin{equation*}
            \L^- = \pm(4 \N + 1/2) \cup \pm (4 \N + 2/3) \cup \{ -2/3, 0 \}, 
        \end{equation*}
        \begin{equation*}
            \L^+ = \pm(4 \N - 1/2) \cup \pm (4 \N - 2/3) \cup \{ -2/3, 0 \},
        \end{equation*}
        \begin{equation*}
            \L^0 = \pm(4 \N + 1/3) \cup \pm (4 \N + 2/3) \cup \{ -2/3, 0 \}.
        \end{equation*}
        Then $\L^-$, $\L^+$ and $\L^0$ are complete interpolating sequences for $PW_E$, but at the same time
        \begin{enumerate}
            \item One point must be added to the sequence $\L^-$ to make it complete interpolating for $PW_{E^g}$,
            \item One point must be removed from the sequence $\L^+$ to make it complete interpolating for $PW_{E^g}$,
            \item Neither adding nor removing a point from the sequence $\L^0$ makes it complete interpolating for $PW_{E^g}$.
        \end{enumerate}
    \end{exm}
    
    There can be nothing more than a single point difference. So, the first two sequences from Example \ref{prop: concrete examples} give the accuracy of the following theorem.
    
    \begin{thrm} \label{thrm: ness +-1}
        Let $\L$ be a complete interpolating sequence for $PW_E$. Then
        \begin{enumerate}
            \item $\L \cup \{ \tilde \l \}$ is a uniqueness set for $PW_{E^g}$ for any $\tilde \l \notin \L$,
            \item $\L \setminus \{ \l_0 \}$ is minimal for $PW_{E^g}$ for any $\l_0 \in \L$.
        \end{enumerate}
    \end{thrm}

    Now let us move on to the properties of complete interpolating sequences for Paley-Wiener spaces on two intervals. If $\L$ is a complete interpolating sequence for $PW_E$, then, as we have seen, a generating function $F$ can be associated with it. As a result, there is a set of extra zeros --- $\T$. 
    
    The set $\T$ turns out to be rather closely related to the Paley-Wiener space $PW_I$ on the complementary interval. It turns out that the set $\T$ is close in its properties to complete interpolating sequences for $PW_I$.
    
    \begin{thrm} \label{thrm: ness T}
        Let $F$ be a generating function for $\L$ in $PW_E$, which has a complete indicator diagram. Let $\T$ be the other zeros of $F$. Suppose that $\L$ is a complete interpolating sequence for $PW_E$. Then a system $\{ \k^I_{t} \}_{t \in \T}$ is complete and minimal in $PW_I$. 
    \end{thrm}     
    
    Thus, Theorem \ref{thrm: ness T} gives a necessary condition for the set $\T$, which corresponds to condition $(ii)$ in Theorem \ref{thrm: suff}. In addition, it is used in the proof of Theorem \ref{thrm: ness +-1}.
    
    {
    We will show (see Lemma \ref{lem: specter}) that for any complete interpolating sequence there exists a generating function with complete indicator diagram. 
    }
    
    Similarly, we will show that the set $\L \cup \T$ is close in its properties to complete interpolating sequences for $PW_{\pi}$. 
    
    \begin{thrm} \label{thrm: ness L T}
        Let $F$ be a generating function for $\L$ in $PW_E$, which has a complete indicator diagram. Suppose that $\L$ is a complete interpolating sequence for $PW_E$.
        Then a system 
        \begin{equation*}
            \{ \k^{\pi}_{t} \}_{t \in \T} \cup \{ \k^{\pi}_{\l} \}_{\l \in \L}    
        \end{equation*}
        is complete and minimal in $PW_{\pi}$.
    \end{thrm}
    
    Thus, we also have a necessary condition which corresponds to condition $(i)$ in Theorem \ref{thrm: suff}. 
    
    Finally, we will show that condition $(iii)$ in Theorem \ref{thrm: suff} is necessary for $\L$ to be complete interpolating for $PW_E$.
    
    \begin{thrm} \label{thrm: ness F+} 
        Let $F$ be a generating function for $\L$ in $PW_E$. Suppose that $\L$ is a complete interpolating sequence for $PW_E$. Then
        \begin{equation*}
            \inf_{t \in \T} \left| \frac{F^+(t)}{F'(t)} \right| || \k^{\pi}_{t} ||_2 ||\k^I_{\bar t}||_2 > 0.
        \end{equation*}
    \end{thrm}
    
    \subsection*{Organization of the paper and notations}
        {
        In Section $2$ we prove some preliminary lemmas. In Section $3$ we study the properties of complete interpolating sequences on two intervals and prove Theorem \ref{thrm: ness T}, Theorem \ref{thrm: ness L T}, Theorem \ref{thrm: ness F+} and Theorem \ref{thrm: ness +-1}. In Section $4$ we prove Theorem \ref{thrm: suff}. In Section $5$ we construct examples of complete interpolating sequences on two intervals.
        
        The notation $U(z) \lesssim V (z)$ (or equivalently $V(z) \gtrsim U(z)$) means that there is a constant $C > 0$ such that $U(z) \le CV (z)$ holds for all $z$ in the set in question, which may be a Hilbert
        space, a set of complex numbers, or a suitable index set. We write $U(z) \asymp V(z)$ if both $U(z) \lesssim V(z)$ and $U(z) \gtrsim V(z)$. We use a notation $\ZZ_F$ for a set of zeros of the function $F$. By $\Span \{v_k\}$ and $\Cl \Span \{ v_k \}$ we denote the linear span and the closed linear span of vectors $\{ v_k \}$ in the corresponding Hilbert space.
        }

\section{Preliminary results}

    { If $F \in PW_\pi$ and $F(\l) = 0$, then $\frac{F(z)}{z - \l}$ also belongs to the space $PW_\pi$. This \textit{division property} is used in the study of complete interpolating sequences for $PW_\pi$.} In particular, this property implies that the generating function of the complete interpolating sequence for $PW_\pi$ is unique up to a multiplicative constant.

    In the case of the space $PW_E$, dividing a function by its zero no longer always gives a result that lies in the same space.

    
    
    \begin{lem} \label{lem: 1}
        Let $F$ be a function from $PW_E + z PW_E$, $\l$ be a zero of $F$. Then
        \begin{equation} \label{eq: div}
            \frac{F(z)}{z - \l} = \f_{\l}(z) + iF^+(\l) \k^I_{\bar \l} (z),
        \end{equation}
        where $\f_{\l} \in PW_E$.
    \end{lem}
    
    \begin{proof}
        \begin{equation*}
            \frac{F(z)}{z - \l} = \frac{F^+(z)}{z - \l} + \frac{F^-(z)}{z - \l} =    
        \end{equation*}
        \begin{equation*}
            \frac{F^+(z) - F^+(\l)e^{ib(z - \l)}}{z - \l} + \frac{F^-(z) - F^-(\l)e^{ia(z - \l)}}{z - \l} + \frac{F^+(\l)e^{ib(z - \l)} + F^-(\l)e^{ia(z - \l)}}{z - \l}.
        \end{equation*}
        Note that 
        \begin{equation*}
            \frac{F^+(z) - F^+(\l)e^{ib(z - \l)}}{z - \l} \in PW_{E^+}, \quad \frac{F^-(z) - F^-(\l)e^{ia(z - \l)}}{z - \l} \in PW_{E^-}.     
        \end{equation*}
        Since $F(\l) = 0$, then $F^-(\l) = -F^+(\l)$, so that
        \begin{equation*}
            \frac{F^+(\l)e^{ib(z - \l)} + F^-(\l)e^{ia(z - \l)}}{z - \l} = F^+(\l) \frac{e^{ib(z - \l)} - e^{ia(z - \l)}}{z - \l} = iF^+(\l) \int_a^b e^{i(z - \l)t} dt = i F^+(\l) \k^I_{\bar \l}(z).    
        \end{equation*}
        Thus, if we put
        \begin{equation*}
            \f_{\l} =  \frac{F^+(z) - F^+(\l)e^{ib(z - \l)}}{z - \l} + \frac{F^-(z) - F^-(\l)e^{ia(z - \l)}}{z - \l} \in PW_E,    
        \end{equation*}
        we get the desirable identity \eqref{eq: div}.
    \end{proof}
    
    For the proof of Theorem \ref{thrm: ness T}, we will need the linear independence of the reproducing kernels.
    \begin{lem} \label{lem: 2}
        Suppose that $\T \subset \C$ and for some $\l \notin \T$ $\k^I_{\l} \in \Cl \Span \{ \k^I_{ t}, \, t \in \T \}$ in $PW_I$.
        
        Then $\Cl \Span \{ \k^I_{ t}, \, t \in \T \} = PW_I$, i.e. $\T$ is a uniqueness set for $PW_I$. 
    \end{lem}
    \begin{proof}
        Assume the opposite, then there is a function
        \begin{equation*}
            h \in \Span \{ \k^I_{ t}, \, t \in \T \}^{\perp}, \quad h \ne 0.
        \end{equation*} 
        In other words, $h(t) = 0, \, t \in \T$. Since the function $h$ is entire and nonzero, then, dividing by $(z - \l)$ to the necessary power, we can assume that $h(\l) \ne 0$. But $\k^I_{\l} \in \Cl \Span \{ \k^I_{ t}, \, t \in \T \}$, so $h \perp \k^I_{\l }$. That is, $h(\l) = (h, \k^I_{\l}) = 0$.
    \end{proof}
    
    In addition, for Theorem \ref{thrm: ness T}, we will need a sufficient condition for the completeness of the mixed system in terms of the completeness radius of the set.
    \begin{definition}[\text{\cite{HavJor}[p. 398]}]
        The radius of completeness of a set $\L$ is
        \begin{equation*}
            R( \L) = \sup\{a > 0: \: \L \text{ is a uniqueness set for } PW_{[0, a]}\}.   
        \end{equation*}
    \end{definition}
    
    \begin{lem}[\cite{invariant C}\text{[Proposition 2.1, Remark 5.2]}] \label{lem: 3}
        
        Suppose $J \subset \R$ is an interval, $G \in PW_J$ has simple zeros, and $G \notin PW_{J'}$ for any proper subinterval $J' \subset J$. If $\ZZ_G = \L_1 \sqcup \L_2$, where $\ZZ_G$ are the zeros of $G$, and $R (\L_2) < |J|$, then the system
        \begin{equation*}
            \left\{  \k^J_\l \right\}_{\l \in \L_2} \cup \left\{ \frac{G(z)}{z - \l} \right\}_{\l \in \L_1}
        \end{equation*}
        is complete in $PW_J$.
    \end{lem}
    
    \begin{lem} \label{lem: lacuna}
        Let $g(z)$ be a function in $PW_{\pi}$, $F(z)$ be a function in $PW_E + zPW_E$ and 
        \begin{equation*}
            g(z) = F(z)u(z),
        \end{equation*}
        where $u(z)$ is of exponential type $0$. Then $g \in PW_E$. 
    \end{lem}
    \begin{proof}
            Note that the function $u$ belongs to the Cartwright class $C$ \cite{Levin}[p. 115]  as a ratio of functions from the class $C$. By the multiplier theorem \cite{Koosis logint2}[p. 397] there exists a function $\f$ of arbitrarily small type and bounded on $\R$ such that $u(x) \f(x) \in L^2(\R)$. Multiplying $\f$ by $\sin (az) e^{icz}$ if necessary, we can assume that $\f$ has an indicator diagram $iI$, so that the system
            \begin{equation} \label{eq: phi}
                \left\{ \frac{\f(x)}{x - z_0} \right\}_{z_0 \in \ZZ_{\f}}    
            \end{equation}
            is complete in $PW_I$.
            
            Thus, it suffices to show that $g$ is orthogonal to the system \eqref{eq: phi}. Let $z_0$ be an arbitrary zero of the function $\f$. Put $u^*(z) = \overline{u(\overline{z})}$.
            
            \begin{equation*} \label{eq: gort}
                \left( g(z), \frac{\f(z)}{z - z_0} \right) = \left( F(z)u(z), \frac{\f(z)}{z - z_0} \right) 
            \end{equation*}
            \begin{equation*}
                = \frac{1}{2 \pi} \int_{\R} F(x)u(x) \overline{\frac{\f(x)}{x - z_0}} dx = \frac{1}{2 \pi} \int_{\R} F(x) \overline{\frac{u^*(x) \f(x)}{x - z_0}} dx
            \end{equation*}
            \begin{equation*}
                = \frac{1}{2 \pi} \int_{\R} (x - \l) f_\l(x) \overline{\frac{u^*(x) \f(x)}{x - z_0}} dx = \frac{1}{2 \pi} \int_{\R} f_\l(x) \overline{\frac{(x - \bar \l) u^*(x) \f(x)}{x - z_0}} dx
            \end{equation*}
            \begin{equation*}
                = \left( f_\l(z), \frac{(z - \bar \l) u^*(z) \f(z)}{z - z_0} \right) = 0, 
            \end{equation*}
            since $u^* \f \in PW_I$ and $f_\l \in PW_E$.
    \end{proof}

\section{Complete interpolating sequences on two intervals}

    \subsection{Basic properties}
        First, let us study the properties of complete interpolating sequences for $PW_E$. It is clear that if $\L$ is a complete interpolating sequence for $PW_E$, then there exists a system $\{ f_\l \}$ biorthogonal to $\{ \k^E_\l \}$ in $PW_E $. That is
        \begin{equation*}
            f_\l(\l') = \de_{\l \l'}, \ \l, \l' \in \L.
        \end{equation*}

        \begin{lem} \label{lem: specter}
             Let $\L$ be a complete interpolating sequence for $PW_E$, and $\{ f_\l \}_{\l \in \L}$ be its biorthogonal system.
             Then at most two biorthogonal elements have an indicator diagram of length less than $2 \pi$.
        \end{lem}
        \begin{proof}
            Since all $f_\l$ lie in $PW_E$, we will show that at most one of them can have spectrum in $[-\pi, c], \ c < \pi$. Similarly, at most one of $f_\l$ can have spectrum in $[d, \pi], \ d > -\pi$. In total, at most two biorthogonal functions can have an indicator diagram of length less than $2\pi$.

            We will prove by contradiction. Let the spectrum of $f_{\l_0}$ and $f_{\l_1}$ be contained in $[-\pi, c], \ c < \pi$. Consider an arbitrary $\l \in \L$.
            
            Let us take two complex numbers $A_\l$ and $B_\l$ not equal to zero at the same time such that
            \begin{equation*}
                A_\l (\l - \l_0)f^+_{\l_0}(\l) + B_\l (\l - \l_1)f^+_{\l_1}(\l) = 0.    
            \end{equation*}
            Let
            \begin{equation*}
                g_\l(z) = C_\l \frac{A_\l (z - \l_0)f_{\l_0}(z) + B_\l (z - \l_1)f_{\l_1}(z)}{z - \l}.
            \end{equation*}
            Then by Lemma \ref{lem: 1} the function $g_\l$ belongs to $PW_E$. Now let us choose $C_\l$ so that $g_\l(\l) = 1$, i.e.
            \begin{equation*}
                C_\l = \frac 1 {A_\l (\l - \l_0)f'_{\l_0}(\l) + B_\l (\l - \l_1)f_{\l_1}(\l)}.
            \end{equation*}
            Thus, $g_\l(\l') = \de_{\l \l'}, \ \l' \in \L$, and hence 
            $g_\l = f_\l$. But then the spectrum of $f_\l$ also lies in $[-\pi, c]$. And since the spectrum of any $f_\l$ lies in $[-\pi, c]$, the system $\{ f_\l \}$ cannot be a Riesz basis in $PW_E$.

        \end{proof}
        
         Denote $F_\l(z) = (z - \l)f_\l(z)$. Such functions will be generating for $\L$ in $PW_E$.
        
        \begin{lem} \label{lem: F^+ zero}
            Let $\L$ be a complete interpolating sequence for $PW_E$ and $F$ be a generating function for $\L$ in $PW_E$. Then there exists $\l \in \L$ such that $F^+(\l) \ne 0$.
        \end{lem}
        \begin{proof}
            Suppose that this is not the case, i.e. $F^+(\l) = 0, \ \l \in \L$. We fix some $\l_0 \in \L$. By Lemma \ref{lem: 1}
            \begin{equation*}
                f_{\l_0}(z) = \frac{F(z)}{z - \l_0} \in PW_E.
            \end{equation*}
            Thus,
            \begin{equation*}
                F(z) = (z - \l_0) f_{\l_0}(z),
            \end{equation*}
            and, therefore,
            \begin{equation*}
                F^\pm(z) = (z - \l_0) f^\pm_{\l_0}(z).
            \end{equation*}
            By our assumption $F^\pm(\l) = F(\l) = 0, \ \l \in \L$. Thus,
            \begin{equation*}
                f^\pm_{\l_0}(\l) = f_{\l_0}(\l) = 0, \ \l \in \L \setminus \{ \l_0 \}.
            \end{equation*}
            Since $f^\pm_{\l_0}, f_{\l_0} \in PW_E$ and $\L$ is a uniqueness set for $PW_E$, then
            \begin{equation*}
                f^\pm_{\l_0}(z) = c^\pm f_{\l_0}(z).
            \end{equation*}
            We know that $f_{\l_0} \ne 0$. Hence, at least one of the numbers $c^+$ and $c^-$ is not equal to $0$. Without loss of generality, we will assume that $c^+ \ne 0$. But then we get that $f_{\l_0} = \frac{1}{c^+} f^+_{\l_0} \in PW_{E^+}$. And thus for any $\l \in \L$
            \begin{equation*}
                f_\l(z) = \frac{F(z)}{z - \l} = \frac{(z - \l_0)f_{\l_0}(z)}{z - \l} \in PW_{E^+}.
            \end{equation*}
            Thus, $\Cl \Span \{f_\l \}_{\l \in \L} \subset PW_{E^+}$. On the other hand, $\L$ is a complete interpolating sequence for $PW_E$, hence the system $\{ f_\l \}$ is a Riesz basis in $PW_E$. We have come to a contradiction.
        \end{proof}
        
        \begin{lem} \label{lem: S(z)}
            Let $\L$ be a complete interpolating sequence for $PW_E$, then there are $\l_1, \l_2 \in \L$ such that the function
            \begin{equation*}
                S(z) = F^+_{\l_1}(z) F^-_{\l_2}(z) - F^-_{\l_1}(z) F^+_{\l_2}(z)
            \end{equation*}
            is not identically equal to $0$. Moreover, in this case $\ZZ_S = \L$.
        \end{lem}
        \begin{proof}
            First, note that for any $\l \in \L$ $F^+_{\l_{1,2}}(\l) = -F^-_{\l_{1,2}}( \l)$, so $S(\l) = 0$. In other words, $\L \subset \ZZ_S$.
            
            Now suppose that $S(z_0) = 0$ for some $z_0 \notin \L$, or $z_0 \in \L$ and $S'(z_0) = 0$. Then by Lemma \ref{lem: 1}
            \begin{equation*}
                \frac{F_{\l_1}(z) F^-_{\l_2}(z_0) - F^-_{\l_1}(z_0) F_{\l_2}(z)}{z - z_0} \in PW_E.
            \end{equation*}
            But on the other hand, this function vanishes on $\L$, which means that it is identically equal to $0$. Similarly, 
            \begin{equation*}
                \frac{F_{\l_1}(z) F^+_{\l_2}(z_0) - F^+_{\l_1}(z_0) F_{\l_2}(z)}{z - z_0} \in PW_E
            \end{equation*}
            vanishes on $\L$ and, therefore, is also identically equal to $0$. Since $\L$ is a uniqueness set for $PW_E$, by Lemma \ref{lem: 1}, $F^+_{\l_2}(z_0)$ and $F^-_{\l_2}(z_0)$ cannot be equal to $0$ at the same time. And, therefore, we can conclude that 
            \begin{equation*}
                F_{\l_1}(z) = c F_{\l_2}(z), \ c \ne 0. 
            \end{equation*}
            Now suppose that the assertion of the lemma is not true. Then, fixing some $\l_0 \in \L$, we obtain for any $\l \in \L$
            \begin{equation*}
                F_{\l}(z) = c_\l F_{\l_0}(z), \ c_\l \ne 0.
            \end{equation*}
            Then, since $F^+_\l(\l) = 0$ for any $\l \in \L$, then
            \begin{equation*}
                F^+_{\l_0}(\l) = 0, \ \l \in \L.
            \end{equation*}
            This contradicts Lemma \ref{lem: F^+ zero}.
        \end{proof}
    
        We are ready to prove Theorems \ref{thrm: ness T}, \ref{thrm: ness L T}, \ref{thrm: ness F+} and \ref{thrm: ness +-1}.
    
    \subsection{Proof of Theorem \ref{thrm: ness T}}
        We will prove the completeness and minimality of the conjugate system $\{ \k^I_{\bar t} \}_{t \in \T}$. We first prove the minimality.  
        
        \textit{Minimality}.
            Suppose that the conjugate system is not minimal. Then there is a $t_0 \in \T$ such that 
            \begin{equation*}
                \k^I_{\bar t_0} \in \Cl \Span \{ \k^I_{\bar t}, \, t \in \T^- \},
            \end{equation*}
            where $\T^- = \T \setminus \{ t_0 \}$. Hence, for any $\de > 0$ there are $t_1, \ldots, t_n \in \T^-$ and $c_1, \ldots, c_n \in \C$ such that
            \begin{equation*}
                \left| \left| \k^I_{\bar t_0} + \sum_{k = 1}^n c_k \k^I_{\bar t_k} \right| \right|_2 < \de.
            \end{equation*}
            Let us denote
            \begin{equation*}
                g_{\de}(z) = \k^I_{\bar t_0}(z) + \sum_{k = 1}^n c_k \k^I_{\bar t_k}(z) \in PW_I,
            \end{equation*}
            and
            \begin{equation*}
                f_{\de}(z) = \sum_{k = 0}^n b_{t_k} \frac{F(z)}{F'(t_k)(z - t_k)},
            \end{equation*}
            \begin{equation*}
                b_{t_k} = \frac{F'(t_k)}{i F^+(t_k)} c_k, \qquad b_{t_0} = \frac{F'(t_0)}{i F^+(t_0)}.
            \end{equation*}
            Note that the definition is correct since $F^+(t) \ne 0, \ t \in \T$. Indeed, if $F^+(t) = 0$, then by Lemma \ref{lem: 1} the function $\frac{F(z)}{z - t} \in PW_E$. This function also vanishes on $\L$. But $\L$ is a uniqueness set for $PW_E$.
            
            The projection of $f_\de$ onto $PW_I$ by Lemma \ref{lem: 1} is equal to
            \begin{equation*}
                \sum_{k = 0}^n b_{t_k} i \frac{F^+(t_k)}{F'(t_k)} \k^I_{\bar t_k}(z) = g_\de(z).
            \end{equation*}
            Hence the projection of $f_\de$ onto $PW_E$ is $f_\de - g_\de \in PW_E$. Since $\L$ is a complete interpolating sequence for $PW_E$, then
            \begin{equation*}
                ||f_\de - g_\de||_2^2 \lesssim \sum_{\l \in \L} \frac{|f_\de(\l) - g_\de(\l)|^2}{|| \k^E_\l ||^2} = \sum_{\l \in \L} \frac{|g_\de(\l)|^2}{|| \k^E_\l ||^2} \lesssim \sum_{\l \in \L} \frac{|g_\de(\l)|^2}{|| \k^{\pi}_\l ||^2}.
            \end{equation*}
            Since $g_\de \in PW_{\pi}$ and $\L$ is interpolating for $PW_\pi$, 
            \begin{equation*}
                \sum_{\l \in \L} \frac{|g_\de(\l)|^2}{|| \k^{\pi}_\l ||^2} \lesssim ||g_\de||_2^2 < \de^2.
            \end{equation*}
            In total, we get the estimate
            \begin{equation*}
                ||f_\de||_2 \le ||f_\de - g_\de||_2 + ||g_\de||_2 \lesssim \de.
            \end{equation*}
            But on the other hand
            \begin{equation*}
                ||f_\de||_2 ||\k^{\pi}_{t_0}||_2 \ge |f_\de(t_0)| = |b_{t_0}| \gtrsim 1,    
            \end{equation*}
            since $b_{t_0}$ is fixed and does not depend on $\de$. Letting $\de$ tend to zero, we get a contradiction.
            
            \textit{Completeness}.
            Let us show that $\overline \T$ is a uniqueness set for $PW_I$. 
    
            By Lemma \ref{lem: F^+ zero}, there exists $\l_0 \in \L$ such that $F^+(\l_0) \ne 0$. Function $F(z)$ has a complete indicator diagram. Radius of completeness $R(\L) \le |E| < 2 \pi$ since the indicator diagram of $S(z)$ lies in $iE^g$ and $|E^g| = |E|$. Hence, by Lemma \ref{lem: 3}
            \begin{equation*}
                PW_{\pi} = \Cl \Span \{ \k_\l \}_{\l \in \L} \oplus \Cl \Span \left\{ \frac{F(z)}{z - t} \right\}_{t \in \T}.
            \end{equation*}
            Let $f = F'(\l_0)f_{\l_0}$, and
            \begin{equation} \label{eq: 2}
                g(z) = f(z) - \frac{F(z)}{z - \l_0}.
            \end{equation}
            It is clear that $g \in PW_\pi$ and $g(\l) = 0, \ \l \in \L$. This is equivalent to $(g, \k_\l) = 0, \ \l \in \L$, hence
            \begin{equation*}
                g \in \Cl \Span \left\{ \frac{F(z)}{z - t} \right\}_{t \in \T}.
            \end{equation*}
            Let us project \eqref{eq: 2} onto $PW_I$ and use Lemma \ref{lem: 1}. Since $f \in PW_E$, we get
            \begin{equation*}
                F^+(\l_0) \k^I_{\bar \l_0} \in \Cl \Span \left\{ \k^I_{\bar t} \right\}_{t \in \T}.
            \end{equation*}
            We know that $F^+(\l_0) \ne 0$. Thus, from Lemma \ref{lem: 2} we get that $\overline \T$ is a uniqueness set for $PW_I$.
        
    \subsection{Proof of Theorem \ref{thrm: ness L T}}
            First of all, the system
            \begin{equation*}
                \{ \k^{\pi}_{t} \}_{t \in \T} \cup \{ \k^{\pi}_{\l} \}_{\l \in \L}
            \end{equation*}
            is minimal in $PW_{\pi}$, since it has the biorthogonal system
            \begin{equation*}
                \left\{ \frac{F(z)}{F'(t)(z - t)} \right\}_{t \in \T} \cup \left\{ \frac{F(z)}{F'(\l)(z - \l)} \right\}_{\l \in \L}.   
            \end{equation*}
            Thus, it remains to show that the set $\L \cup \T$ is a uniqueness set for $PW_{\pi}$. 
            Suppose that this is not true. That means that there is a function $g(z) \in PW_{\pi}$, such that
            \begin{equation*}
                g(\l) = 0, \ \l \in \L \qquad g(t) = 0, \ t \in \T. 
            \end{equation*}
            Hence, we can factorize
            \begin{equation*}
                g(z) = F(z)u(z).
            \end{equation*}
            Note that $u(z)$ is of exponential type $0$, since $g(z)$ is of type $\pi$ and $F(z)$ has a complete indicator diagram. Therefore, by Lemma \ref{lem: lacuna}, $g(z) \in PW_E$, which means that $\L$ is not a uniqueness set for $PW_E$. We have arrived at a contradiction.
    \subsection{Proof of Theorem \ref{thrm: ness F+}}
            Suppose that the opposite is true. Then for any $\e > 0$ there is $t_\e \in \T$ such that 
            \begin{equation*}
                \left| \frac{F^+(t_\e)}{F'(t_\e)} \right| < \e || \k^{\pi}_{ t_\e} ||_2^{-1} ||\k^I_{\bar t_\e}||_2^{-1}.    
            \end{equation*}
            Consider a function
            \begin{equation*}
                f_\e(z) = \frac{F(z)}{F'(t_\e)(z - t_\e)}.
            \end{equation*}
            By Lemma \ref{lem: 1} its projection on $PW_I$ is
            \begin{equation*}
                P_I f_\e (z) = i \frac{F^+(t_\e)}{F'(t_\e)} \k^I_{\bar t_\e}.
            \end{equation*}
            Hence,
            \begin{equation} \label{eq: P_I f < e}
                || P_I f_\e ||_2 < \e || \k^{\pi}_{ t_\e} ||_2^{-1}.    
            \end{equation} 
            Thus, since $\L$ is interpolating for $PW_E$ and consequently for $PW_\pi$, 
            \begin{equation*}
                \left| \left| \frac{P_I f_\e (\l)}{|| \k^{\pi}_{\l} ||} \right| \right|_{l^2(\L)} \lesssim || P_I f_\e ||_2 \lesssim \e || \k^{\pi}_{ t_\e} ||_2^{-1}.
            \end{equation*}
            Note that $f_\e(\l) = 0, \ \l \in \L$. Hence,
            \begin{equation*}
                \left| \left| \frac{(f_\e - P_I f_\e) (\l)}{|| \k^{E}_{\l} ||} \right| \right|_{l^2(\L)} \lesssim \left| \left| \frac{(f_\e - P_I f_\e) (\l)}{|| \k^{\pi}_{\l} ||} \right| \right|_{l^2(\L)} \lesssim \e || \k^{\pi}_{ t_\e} ||_2^{-1},
            \end{equation*}
            which implies
            \begin{equation} \label{eq: P_E f_e < e}
                || f_\e - P_I f_\e ||_2 \lesssim \e || \k^{\pi}_{ t_\e} ||_2^{-1},
            \end{equation}
            since $f_\e - P_I f_\e \in PW_E$ and $\L$ is complete interpolating for $PW_E$. Thus, from \eqref{eq: P_I f < e} and \eqref{eq: P_E f_e < e}, by the triangle inequality, we conclude that
            \begin{equation*}
                || f_\e ||_2 \lesssim \e || \k^{\pi}_{ t_\e} ||_2^{-1}.
            \end{equation*}
            But at the same time, $f_\e(t_\e) = 1$. Thus,
            \begin{equation*}
                1 = |f_\e(t_\e)| = |(f_\e, \k^{\pi}_{t_\e})| \le || f_\e ||_2 || \k^{\pi}_{t_\e} ||_2 \lesssim \e,
            \end{equation*}
            so, choosing sufficiently small $\e$, we arrive at the contradiction.
        
    \subsection{Proof of Theorem \ref{thrm: ness +-1}}
            By Lemma \ref{lem: specter}, take $\l \in \L$ such that $F(z) = (z - \l)f_{\l}(z)$ has a complete indicator diagram.
            
            Let us start with $(1)$.
            
            Suppose that $\L \cup \{ \tilde{\l} \}$ is not a uniqueness set for $PW_{E^g}$ for some $\tilde \l \notin \L$. Then there exists an entire function $H(z)$ in $PW_{E^g}$ such that $H(\l) = 0, \ \l \in \L \cup \{ \tilde \l \}$.
            
            By Theorem \ref{thrm: ness T} $\T$ is minimal for $PW_I$, so for any $t \in \T$, there exists a function $f_t \in PW_I$ such that $f_t(t') = \de_{t t'}, \ t' \in \T$. Then we can conclude that
            \begin{equation} \label{eq: HG = F}
                \frac{f_t(z) (z - t) H(z)}{z - \tilde \l} = F(z) U(z).
            \end{equation}
            The indicator diagram of the function on the left has a length not exceeding $|E| + |I| = 2\pi$. Hence, due to the completeness of the diagram of $F$, the indicator diagram of $U$ has zero length.
            That is, $U(z) = e^{izc}u(z)$, $c \in \R$, $u(z)$ is an entire function of type $0$. Multiplying \eqref{eq: HG = F} by the exponent, we get
            \begin{equation*}
                \frac{f_t(z) (z - t) H(z)}{z - \tilde \l} e^{-izc} = F(z) u(z).
            \end{equation*}
            Let us denote
            \begin{equation*}
                g(z) = \frac{f_t(z) (z - t) H(z)}{z - \tilde \l} e^{-izc} = F(z) u(z).
            \end{equation*}
            Note that $f_t \in L^2(\R)$ and $H(x) \in L^2(\R)$, so $\frac{(x - t)H(x)}{x - \tilde \l} \in L^2(\R)$. Hence, by the Cauchy-Schwarz inequality $g(x) \in L^1(\R)$. But since $g$ also has a finite type, $g$ is bounded on $\R$. Hence, $g(x) \in L^2(\R)$. In total, we conclude that $g \in PW_\pi$.
            Thus, by Lemma \ref{lem: lacuna}, $g \in PW_E$.
            
            Note that $g(\l) = 0, \ \l \in \L$. So we came to a contradiction with the fact that $\L$ is a uniqueness set for $PW_E$.

            Let us proceed to $(2)$.
            
            Let $\l_0 \in \L$. By Lemma \ref{lem: 2} $\L \setminus \{ \l_0 \}$ is minimal for $PW_{E^g}$ if and only if for any $\l_1 \in \L \setminus \{ \l_0 \}$ the set $\L \setminus \{ \l_0, \l_1 \}$ is not a uniqueness set for $PW_{E^g}$.
            
            By definition
            \begin{equation*}
                 S(z) = F^+_{\l_3}(z) F^-_{\l_4}(z) - F^+_{\l_4}(z) F^-_{\l_3}(z), \ \l_3, \l_4 \in \L.  
            \end{equation*}
            Let us show that the function 
            \begin{equation*}
                g(z) = \frac{S(z)}{(z - \l_0)(z - \l_1)} = \frac{F^+_{\l_3}(z) F^-_{\l_4}(z) - F^+_{\l_4}(z) F^-_{\l_3}(z)}{(z - \l_0)(z - \l_1)}    
            \end{equation*}
            lies in $PW_{E^g}$, and then we will get that $\L \setminus \{ \l_0, \l_1 \}$ is not a uniqueness set for $PW_{E^g}$, since $g(\l) = 0, \ \l \in \L \setminus \{ \l_0, \l_1 \}$.
            
            It is clear that the indicator diagram of $S$ is a subset of $iE^g$. Hence, the indicator diagram of $g$ also lies $iE^g$, so it suffices to show that $g(x) \in L^1(\R)$. And this is true, since
            \begin{equation*}
                g(z) = \frac{(z - \l_3)(z - \l_4) }{(z - \l_0)(z - \l_1)} (f^+_{\l_3}(z) f^-_{\l_4}(z) - f^+_{\l_4}(z) f^-_{\l_3}(z)),
            \end{equation*}
            and $f^\pm_{\l_3}, f^\pm_{\l_4} \in L^2(\R)$.

\section{Sufficient conditions for completeness and interpolation on two intervals}
    \subsection{Proof of Theorem \ref{thrm: suff}}
            First we will prove that $\L$ is a uniqueness set for $PW_E$.
            Suppose $f \in PW_E$ is a function such that $f(\l) = 0, \, \l \in \L$. Then $f \in PW_{\pi}$, hence,
            \begin{equation*}
                f(z) = \sum_{\l \in \L} \frac{F(z)}{F'(\l)(z - \l)} f(\l) + \sum_{t \in \T} \frac{F(z)}{F'(t)(z - t)} f(t) 
            \end{equation*}
            \begin{equation*}
                = \sum_{t \in \T} \frac{F(z)}{F'(t)(z - t)} f(t).
            \end{equation*}
            Then by Lemma \ref{lem: 1}
            \begin{equation*}
                f(z) = \sum_{t \in \T} \left( f_t(z) + iF^+(t) \k^I_{\bar t} (z) \right) \frac{f(t)}{F'(t)}.
            \end{equation*}
            Since $f \in PW_E$, then projecting it onto $PW_I$, we get $0$, i.e.
            \begin{equation} \label{eq: sum k_t}
                \sum_{t \in \T} F^+(t) \frac{f(t)}{F'(t)} \k^I_{\bar t} (z) = 0. 
            \end{equation}
            Since $\T$ is a complete interpolating sequence for $PW_I$, the system $\left\{ \frac{\k^I_t}{|| \k^I_t ||_2} \right\}_{t \in \T}$ is a Riesz basis in $PW_I$. Hence, the conjugate system $\left\{ \frac{\k^I_{\bar t}}{|| \k^I_{\bar t} ||_2} \right\}_{t \in \T}$ is also a Riesz basis in $PW_I$. Thus, from \eqref{eq: sum k_t} it follows that
            \begin{equation*}
                F^+(t) \frac{f(t)}{F'(t)} = 0, \ t \in \T.
            \end{equation*}
            From \eqref{eq: F^+ > F'} it follows that $F^+(t) \ne 0, \ t \in \T$, so that
            \begin{equation*}
                f(t) = 0, \ t \in \T.
            \end{equation*}
            Thus, $f = 0$.
            
            Now let $\{ a_{\l} \}$, such that
            \begin{equation*}
                \sum_{\l \in \L} \frac{|a_\l|^2}{|| \k^E_\l ||_2^2} < \infty,    
            \end{equation*}
            be given. We want to find a function $f \in PW_E$ such that $f(\l) = a_\l, \ \l \in \L$ and
            \begin{equation} \label{eq: f < a E}
                || f ||_2^2 \lesssim \sum_{\l \in \L} \frac{|a_\l|^2}{|| \k^E_\l ||_2^2}.
            \end{equation}
            Using
            \begin{equation*}
                || \k^E_\l ||_2^2 \asymp e^{2 \pi |\Im \l|} (1 + |\Im \l|)^{-1} \asymp || \k^{\pi}_\l ||_2^2
            \end{equation*}
            we get that \eqref{eq: f < a E} is equivalent to
            \begin{equation} \label{eq: f < a}
                || f ||_2^2 \lesssim \sum_{\l \in \L} \frac{|a_\l|^2}{|| \k^\pi_\l ||_2^2}.    
            \end{equation}
            Let us look for $f$ in the form $f = g + h$, where
            \begin{equation*}
                g(z) = \sum_{\l \in \L} \frac{F(z)}{F'(\l)(z - \l)} a_{\l}, \qquad h(z) = \sum_{t \in \T} \frac{F(z)}{F'(t)(z - t)} b_t,
            \end{equation*}
            with 
            \begin{equation*}
                \sum_{t \in \T} \frac{|b_t|^2}{||\k^\pi_t||_2^2} < \infty.
            \end{equation*}
            Note that for any such $\{ b_t \}_{t \in \T}$ we already have $f(\l) = a_\l, \ \l \in \L$ and $f \in PW_\pi$, since $\L \cup \T$ is complete interpolating for $PW_\pi$. So that we only need to guarantee that projection of $f$ on $PW_I$ is $0$, i.e.,
            \begin{equation} \label{eq: P_I f}
                P_I f  = P_I g + P_I h = 0.
            \end{equation}
            By Lemma \ref{lem: 1} we can expand the equation \eqref{eq: P_I f} as
            \begin{equation*}
                P_I g + i \sum_{t \in \T} \frac{F^+(t)}{F'(t)} b_t \k^I_{\bar t} = 0,
            \end{equation*}
            or, alternatively,
            \begin{equation} \label{eq: iP_I g =}
                \sum_{t \in \T} \frac{F^+(t)}{F'(t)} b_t \k^I_{\bar t} = i P_I g. 
            \end{equation}
            As we already mentioned, $\left\{ \frac{\k^I_{\bar t}}{|| \k^I_{\bar t} ||_2} \right\}_{t \in \T}$ is a Riesz basis in $PW_I$, hence, the equation \eqref{eq: iP_I g =} has a unique solution $\{ b_t \}_{t \in \T}$. Moreover 
            \begin{equation*}
                || P_I g||_2^2 = \left| \left| \sum_{t \in \T} \frac{F^+(t)}{F'(t)} b_t \k^I_{\bar t} \right| \right|_2^2 \asymp \sum_{t \in \T} \left| \frac{F^+(t)}{F'(t)} \right|^2 |b_t|^2 ||\k^I_{\bar t}||_2^2.    
            \end{equation*}
            Thus, by using \eqref{eq: F^+ > F'}, we arrive at
            \begin{equation*}
                \sum_{t \in \T} \frac{|b_t|^2}{||\k^\pi_{t}||_2^2} \lesssim || P_I g ||_2^2.
            \end{equation*}
            It remains to notice that
            \begin{equation*}
                || P_I g ||_2^2 \le ||g||_2^2 \lesssim \sum_{\l \in \L} \frac{|a_\l|^2}{|| \k^\pi_\l ||_2^2}.    
            \end{equation*}
            Hence,
            \begin{equation*}
                \sum_{t \in \T} \frac{|b_t|^2}{||\k^\pi_{t}||_2^2} \lesssim \sum_{\l \in \L} \frac{|a_\l|^2}{|| \k^\pi_\l ||_2^2}.
            \end{equation*}
            So that,
            \begin{equation*}
                || h ||_2^2 \lesssim \sum_{\l \in \L} \frac{|a_\l|^2}{|| \k^\pi_\l ||_2^2}.
            \end{equation*}
            Finally,
            \begin{equation*}
                || g ||_2^2 \lesssim \sum_{\l \in \L} \frac{|a_\l|^2}{|| \k^\pi_\l ||_2^2}.    
            \end{equation*}
            By triangle inequality, we arrive at
            \begin{equation*}
                || f ||_2^2 \lesssim \sum_{\l \in \L} \frac{|a_\l|^2}{|| \k^\pi_\l ||_2^2}.
            \end{equation*}

\section{Examples}
    \subsection{Two intervals of equal length. \label{sbc: examples}}
    
        \begin{thrm}
            Suppose $E = [-\pi, -\pi/2] \cup [\pi/2, \pi]$,
            and $H$ is a generating function for $\ZZ_H$, which is a complete interpolating sequence for $PW_{\pi/4}$. Suppose that the points of $\ZZ_H$ are uniformly separated from the set $\frac 4 3 \Z$.
            Then $\L = \ZZ_H \cup 4 \Z$ is a complete interpolating sequence for $PW_E$. 
        \end{thrm}
        \begin{proof}
            Consider a function
            \begin{equation*}
                F(z) = H(z) \sin \left( \frac 3 4 \pi z \right).
            \end{equation*}
            By \cite{LubSeip}[Theorem 1], zeros of $F(z)$ form a complete interpolating sequence for $PW_{\pi}$. Note that the set of zeros is $\ZZ_F = \ZZ_H \cup \frac 4 3 \Z$. It is clear that $F \in PW_E + z PW_E$ and
            \begin{equation*}
                F^{\pm}(z) = \frac {H(z)} 2 e^{\pm i \frac 3 4 \pi z}.
            \end{equation*}
            Consider a factorization $F(z) = S(z)G(z)$, where 
            \begin{equation*}
                S(z) = H(z) \sin \left( \frac \pi 4 z \right), \ G(z) = 1 + 2 \cos \left( \frac \pi 2 z \right).
            \end{equation*}
            The zeros of $G(z)$ form a complete interpolating sequence for $PW_{\pi/2}$. Put $\T = \ZZ_G = \pm \frac 4 3 + 4 \Z$. Then $|F^+(t)| = |H(t)|/2 \asymp |F'(t)|, \ t \in \T$.
            
            Thus, from Theorem \ref{thrm: suff}, $\L = \ZZ_S$ is a complete interpolating sequence for $PW_E$.
        \end{proof}
        
        \begin{rmk} \label{rmk: crucial (iii)}
            If we exchange a point $\l_0$ from $\ZZ_H$ with a point $t_0$ from $T$, then conditions $(i), (ii)$ of Theorem \ref{thrm: suff} will be still satisfied. But since $F^+(\l_0) = 0$, then, by Lemma \ref{lem: 1}, $\frac{F(z)}{z - \l_0}$ belongs to  $PW_E$ while vanishing on $\L' = \L \setminus \{ \l_0 \} \cup \{ t_0 \}$. Thus, $\L'$ is not a complete interpolating sequence for $PW_E$. 
        \end{rmk}
    
    \subsection{``Extra point'' effect}
    \begin{thrm} \label{thrm: suff +-1}
        Suppose we are in the conditions of Theorem \ref{thrm: suff}. In addition, let
        \begin{equation} \label{eq: pow}
            \left| \frac{F(x)}{\dist(x, \L \cup \T)} \right|^2 \asymp (1 + |x|)^\a, \qquad \left| \frac{G(x)}{\dist(x, \T)} \right|^2 \asymp (1 + |x|)^\b,
        \end{equation}
        where $G$ is an entire function with zeros $\T$, $\a, \b \in \R$.
        
        Then $\a, \b \in (-1, 1)$ and there are 4 cases:
        \begin{enumerate}
            \item $\a - \b \in (-1, 1)$, then $\L$ is a complete interpolating sequence for $PW_{E^g}$;
            \item $\a - \b \in (-2, -1)$, then for any $\tilde{\l} \notin \L$, $\L \cup \{ \tilde{\l} \}$ is a complete interpolating sequence for $PW_{E^g}$;
            \item $\a - \b \in (1, 2)$, then for any $\l_0 \in \L$, $\L \setminus \{ \l_0 \}$ is a complete interpolating sequence for $PW_{E^g}$;
            \item $|\a - \b| = 1$, then neither removing nor adding a point to the set $\L$ makes it a complete interpolating sequence for $PW_{E^g}$.  
        \end{enumerate}
        At the same time, $\L$ is in any case a complete interpolating sequence for $PW_E$.
    \end{thrm}
    \begin{proof}
        We already know that $\L$ is a complete interpolating sequence for $PW_E$ from Theorem \ref{thrm: suff}.
    
        First, note that the weight $v(x) = (1 + |x|)^c$ lies in the Muckenhoupt class $A_2$ if and only if $c \in (-1, 1)$.
        
        Since $\T$ and $\L \cup \T$ are complete interpolating sequences for $PW_I$ and $PW_\pi$, respectively, then by \cite{LubSeip}[Theorem 1] 
        \begin{equation*}
            \left| \frac{G(x)}{\dist(x, \T)} \right|^2 \in A_2, \qquad \left| \frac{F(x)}{\dist(x, \L \cup \T)} \right|^2 \in A_2,
        \end{equation*}
        Hence, taking into account \eqref{eq: pow}, we get that $\a, \b \in (-1, 1)$.
        
        Now let us show that
        \begin{equation*}
            \left| \frac{S(x)}{\dist(x, \L)} \right|^2 \asymp (1 + |x|)^{\a - \b}.
        \end{equation*}
        Indeed, since $\L \cup \T$ is separated and $\T$ is relatively dense, then
        \begin{equation*}
            \frac{\dist(x, \L \cup \T)}{\dist(x, \T)} \asymp \dist(x, \L).
        \end{equation*}
        Hence, taking into account that $F(x) = S(x)G(x)$, we get
        \begin{equation*}
            \left| \frac{S(x)}{\dist(x, \L)} \right|^2 \asymp \left| \frac{F(x)}{\dist(x, \L \cup \T)} \frac{\dist(x, \T)}{G(x)} \right|^2 \asymp (1 + |x|)^{\a - \b}.
        \end{equation*}
        Thus, it remains to consider the cases $(1) \,$--$\, (4)$ and apply \cite{LubSeip}[Theorem 1].
        
        $1)$ If $\a - \b \in (-1, 1)$, then
        \begin{equation*}
            \left| \frac{S(x)}{\dist(x, \L)} \right|^2 \in A_2,
        \end{equation*}
        and hence $\L$ is a complete interpolating sequence for $PW_{E^g}$.
        
        $2)$ If $\a - \b \in (-2, -1)$, then it is clear that for any $\tilde{\l} \notin \L$ we have
        \begin{equation*}
            \left| \frac{S(x)(x - \tilde{\l})}{\dist(x, \L \cup \{ \tilde{\l} \})} \right|^2 \asymp (1 + |x|)^{\a - \b + 2}, 
        \end{equation*}
        and since $\a - \b + 2 \in (0, 1)$, then $\L \cup \{ \tilde{\l} \}$ is a complete interpolating sequence for $PW_{E^g }$.
        
        $3)$ If $\a - \b \in (1, 2)$, then for any $\l_0 \in \L$ we have
        \begin{equation*}
            \left| \frac{S(x)}{(x - \l_0)\dist(x, \L \setminus \{ \l_0 \})} \right|^2 \asymp (1 + |x|)^{\a - \b - 2},
        \end{equation*}
        and since $\a - \b - 2 \in (-1, 0)$, then $\L \setminus \{ \l_0 \}$ is a complete interpolating sequence for $PW_{E^g}$.

        $4)$ If $|\a - \b| = 1$, then $\a - \b$ and $\a - \b \pm 2$ do not lie in $(-1, 1)$. Hence, neither adding nor removing a point makes the set $\L$ a complete interpolating sequence for $PW_{E^g}$.
    \end{proof}

    Thus, using Theorem \ref{thrm: suff +-1}, one can obtain interesting examples of sets $\L$. Namely, such sets $\L$ that are complete interpolating sequences for $PW_E$, but at the same time for $PW_{E^g}$ the set $\L$ either lacks one more point to become complete interpolating, or vice versa one point from the set $\L$ must be removed for it to become complete interpolating.
    
    To build examples, we will need a function whose zeros are shifted points of $\Z$.
    
    \begin{definition}
            Let $|\de| < 1/4$. Denote
            \begin{equation*}
                f_{\de}(z) = z \prod_{k \in \N} \left( 1 - \frac{z^2}{(k + \de)^2} \right).
            \end{equation*}
    \end{definition}
    
    It is well known that $f_{\de}(z)$ is an entire function of type $\pi$ satisfying an asymptotic relation
    \begin{equation}
        \left | \frac{f_{\de}(x)}{\dist (x, \ZZ_{f_{\de}})} \right| \asymp (1 + |x|)^{-2 \de}.    
    \end{equation}
    
    Let $E = [-\pi, -\pi/2] \cup [\pi/2, \pi]$. Denote $g_{\de}(z) = f_{\de}(\frac{z}4)$, and
    \begin{equation} \label{eq: F}
        F(z) = g_{\de}(z) \cos \left( \frac{3\pi}4 z \right).
    \end{equation}

    \begin{prop} \label{prop: 3}
        If $\de > 1/12$ and $\de \ne 1/6$, then there exists a factorization $F(z) = S(z)G(z)$ that satisfies the conditions of Theorem \ref{thrm: suff +-1} case $(2)$, and thus the set $\L$ lacks one point to become a complete interpolating sequence for $PW_{E^g}$.
    \end{prop}
    
    \begin{prop} \label{prop: 4}
        If $\de < -1/12$ and $\de \ne -1/6$, then there exists a factorization $F(z) = S(z)G(z)$ that satisfies the conditions of Theorem \ref{thrm: suff +-1} case $(3)$, and thus one point must be removed from the set $\L$ for it to become a complete interpolating sequence for $PW_{E^g}$.
    \end{prop}

    \begin{proof} 
        We will simultaneously prove both propositions.
        We know that
        \begin{equation*}
            g_{\de}(z) = \frac{z}4 \prod_{k \in \N} \left( 1 - \frac{z^2}{(4k + 4\de)^2} \right)
        \end{equation*}
        and
        \begin{equation*}
            \cos \left( \frac{3\pi}4 z \right) = \prod_{k \in \N_0} \left( 1 - \frac{z^2}{\left( \frac2 3 + \frac4 3 k \right)^2}\right) 
        \end{equation*}
        \begin{equation*}
            = \prod_{k \in \N_0} \left( 1 - \frac{z^2}{\left( \frac2 3 + 4 k \right)^2}\right) \prod_{k \in \N_0} \left( 1 - \frac{z^2}{\left( 2 + 4 k \right)^2}\right) \prod_{k \in \N} \left( 1 - \frac{z^2}{\left( -\frac2 3 + 4 k \right)^2}\right).
        \end{equation*}
        
        Note that the zeros $g_{\de}$ and $\cos \left( \frac{3\pi}4 z \right)$ are separated for $|\de| \ne $1/6. Thus,
        \begin{equation} \label{eq: Fas}
            \left | \frac{F(x)}{\dist (x, \ZZ_F)} \right|^2 \asymp (1 + |x|)^{-4 \de}.     
        \end{equation}
        Since $|4 \de| < 1$, we get that $\left | \frac{F(x)}{\dist (x, \ZZ_F)} \right|^2 \in A_2$. The function $g_{\de}$ is of type $\frac{\pi}4$, so $F$ is of type $\pi$.
        
        Thus \cite{LubSeip}[Theorem 1], $\ZZ_F$ is a complete interpolating sequence for $PW_{\pi}$.
        
        Now we will be proving only Proposition \ref{prop: 3}. Denote
        \begin{equation*}
            G(z) = \left( z - \frac 2 3 \right) \prod_{k \in \N_0} \left( 1 - \frac{z^2}{\left( 2 + 4 k \right)^2}\right) \prod_{k \in \N} \left( 1 - \frac{z^2}{\left( -\frac2 3 + 4 k \right)^2}\right) =
        \end{equation*}
        \begin{equation*}
            = 4\frac{z - \frac 2 3}{z} \cos \left( \frac{\pi}4 z \right) \frac z 4 \prod_{k \in \N} \left( 1 - \frac{\left( \frac z 4 \right) ^2}{\left( -\frac1 6 + k \right)^2}\right) = 4 \frac{z - \frac 2 3}{z} \cos \left( \frac{\pi}4 z \right) f_{-1/6} \left( \frac z 4 \right).
        \end{equation*}
        Hence, $G(z)$ is of type $\frac \pi 2$ and
        \begin{equation} \label{eq: T prop 3}
            \left | \frac{G(x)}{\dist (x, \T)} \right| \asymp (1 + |x|)^{2 \tilde{\de}}, \quad \tilde{\de} = 1/6.
        \end{equation}
        Thus, \cite{LubSeip}[Theorem 1], $\T = \ZZ_G$ is a complete interpolating sequence for $PW_{\pi / 2}$.
        
        If $\de > 1/12$, then $4\de + 4 \tilde{\de} > 1$. Hence, by Theorem \ref{thrm: suff +-1} we get that the sequence $\L$ lacks one more point to become complete interpolating for $PW_{E^g}$. For Proposition \ref{prop: 3}, it only remains to prove that the Hilbert operator $\H_{\L, \T}$ is bounded.
        
        Let us move on to the proof of Proposition \ref{prop: 4}. Denote 
        \begin{equation*}
            G(z) = \left( z - \frac 2 3 \right) \prod_{k \in \N_0} \left( 1 - \frac{z^2}{\left( 2 + 4 k \right)^2}\right) \prod_{k \in \N} \left( 1 - \frac{z^2}{\left( \frac2 3 + 4 k \right)^2}\right)
        \end{equation*}
        \begin{equation*}
            = 4 \frac{z - \frac 2 3}{z} \cos \left( \frac{\pi}4 z \right) f_{1/6} \left( \frac z 4 \right)
        \end{equation*}
        $G(z)$ is of type $\frac \pi 2$ and
        \begin{equation} \label{eq: T prop 4}
            \left | \frac{G(x)}{\dist (x, \T)} \right| \asymp (1 + |x|)^{2 \tilde{\de}}, \quad \tilde{\de} = -1/6.
        \end{equation}
        Thus, $\T = \ZZ_G$ is a complete interpolating sequence for $PW_{\pi / 2}$. If $\de < -1/12$, then $4 \de + 4 \tilde{\de} < -1$. Hence, by Theorem \ref{thrm: suff +-1}, any point must be removed from the sequence $\L$ in order for it to become complete interpolating for $PW_{E^g}$. For Proposition \ref{prop: 4}, it remains only to prove that the Hilbert operator $\H_{\L, \T}$ is bounded.
        
        To complete the proof of Proposition \ref{prop: 3} and Proposition \ref{prop: 4} it remains to show that
        \begin{equation*}
            |F^+(t)| \gtrsim |F'(t)|, \ t \in \T.
        \end{equation*}
        From \eqref{eq: F} we get that
        \begin{equation*}
            F^+(z) = \frac{e^{i  \frac{3 \pi}4 z }}2 g_\de(z).
        \end{equation*}
        Hence,
        \begin{equation*}
            |F^+(t)|^2 \asymp |g_\de(t)|^2 \asymp (1 + |t|)^{-4 \de}, \ t \in \T.
        \end{equation*}
        Next, from \eqref{eq: Fas} we get that
        \begin{equation*}
            |F'(t)|^2 \asymp (1 + |t|)^{-4 \de}, \ t \in \T.
        \end{equation*}
        Therefore, 
        \begin{equation*}
            |F^+(t)| \gtrsim |F'(t)|, \ t \in \T.    
        \end{equation*}
    \end{proof}
    
    Considering $\de = 1/8$, $\de = -1/8$ and $\de = 1/12$, we get the sets $\L^-$, $\L^+$ and $\L^0$ from Example \ref{prop: concrete examples}. Hence, the part of Example \ref{prop: concrete examples} regarding $\L^\pm$ is a direct consequence of Proposition \ref{prop: 3} and Proposition \ref{prop: 4}. For $\L^0$, $\de = 1/12$, it suffices to note that the same arguments as in the proof of Proposition \ref{prop: 3}, by Theorem \ref{thrm: suff +-1}, give the required property for $\L^0$.

\end{document}